\documentclass[11pt]{elsarticle}
\usepackage{amssymb}
\usepackage{mathrsfs}
\usepackage{hyperref}
\usepackage{lmodern}

\newtheorem{thm}{Theorem}[section]
\newtheorem{cor}[thm]{Corollary}
\newtheorem{lem}[thm]{Lemma}

\newtheorem{prop}[thm]{Proposition}
\newtheorem{example}[thm]{Example}

\newcommand{\x}{\mathbf{x}}
\newcommand{\ga}{\mathbf{a}}
\newcommand{\gb}{\mathbf{b}}

\newcommand{\G}{\mathscr{G}}
\newcommand{\D}{\mathscr{D}}

\newcommand{\im}{{\rm Im\,}}
\newcommand{\si}{\sigma}

\newcommand{\de}{\delta}

\newproof{pf}{Proof}

\journal{Journal of Pure and Applied Algebra}

\begin{document}

\begin{frontmatter}

\title{{\bf Locally nilpotent skew extensions of rings}}%

\author{Piotr Grzeszczuk}%
\fnref{fn1}

\ead{p.grzeszczuk@pb.edu.pl}

\fntext[fn1]{The research of was supported by
the Bialystok University of Technology grant WZ/WI/1/2019 and funded by the
resources for research by Ministry of Science and Higher Education of Poland.}

\address{Faculty of Computer Science, Bialystok University of
Technology, Wiejska 45A, 15-351 Bia{\l}ystok, Poland}

\begin{abstract}
We extend existing results on locally nilpotent differential polynomial rings 
to skew extensions of rings. We prove that if $\mathscr{G}=\{\sigma_t\}_{t\in T}$ is a locally finite family of automorphisms of an algebra $R$, $\mathscr{D}=\{\delta_t\}_{t\in T}$ is a family of skew derivations of $R$ such that the prime radical $P$ of $R$ is strongly invariant under $\D$, then the ideal  $P\langle T,\mathscr{G},\mathscr{D}\rangle^*$ of $R\langle T,\mathscr{G},\mathscr{D}\rangle$, generated by $P$, is locally nilpotent.  We then apply this result to algebras with locally nilpotent derivations. We prove that any algebra $R$ over a field of characteristic $0$, having a surjective locally nilpotent derivation $d$ with commutative kernel, and such that $R$ is generated by $\ker d^2$, has  a locally nilpotent Jacobson radical. 
\end{abstract}

\begin{keyword}
locally nilpotent ring \sep skew derivation \sep skew extension \sep Jacobson radical
\MSC[2010] 16N20  \sep 16N40 \sep 16W25   
\end{keyword}
\end{frontmatter}

\section*{Introduction}

Recently, there has been a significant interest in some problems concerning the local nilpotency and other radical properties of differential polynomial rings (\cite{BMS,C, GSZ, S, SZ}). We recall that a ring $R$ is locally nilpotent if every finitely generated subring of $R$ is a nilpotent ring. A. Smoktunowicz and M. Ziembowski in \cite{SZ} negatively answered a question of Shestakov, by constructing an example of a locally nilpotent
ring $R$ such that $R[x;\delta]$ is not equal to its own Jacobson radical. On the other hand, M. Chebotar in \cite{C} proved that the differential polynomial ring $R[x;\delta]$ over locally nilpotent ring $R$ cannot be mapped onto a ring with a non-zero idempotent; in the language of the radical theory it means that if $R$ is locally nilpotent, 
then $R[x;\delta]$ is always Behrens radical. In \cite{BMS} J. Bell, B. Madill and F. Shinko  proved that if $R$ is a locally nilpotent ring satisfying a polynomial identity, then every skew polynomial
ring $R[x;\delta]$ is locally nilpotent.  Next,  B. Greenfeld, A. Smoktunowicz and M. Ziembowski extended this result  in \cite{GSZ}, where they proved that if $R$ is a Baer radical ring, then every differential polynomial ring $R[x;\delta]$ is locally nilpotent. Our main goal is to extend  some of the quoted above results on locally nilpotent differential polynomial rings to skew extensions of rings. We now introduce the notion of a skew extension of a ring.

For a given ring $R$ with an automorphism 
$\sigma$, a {\it $\sigma$-derivation} of $R$ is an additive map $\delta$ from $R$ to itself satisfying the Leibniz rule: $\delta(xy)=\delta(x)y+\sigma(x)\delta(y) $, for $x,y\in R$.  An extension $R\subseteq S$ of rings is said to be a {\it skew extension} if $S$ contains a subset $T$ such that 
\begin{enumerate}
\item[(a)] $S$ is generated as a ring by $R\cup T$;
\item[(b)] for any $t\in T$ there exists a ring automorphism $\sigma_t$ of $R$ and a $\sigma_t$-derivation $\delta_t\colon R\to R$, such that $tr=\sigma_t(r)t+\delta_t(r)$ for all $r\in R$.
\end{enumerate} 
The ring $S$ satisfying conditions (a) and (b)  will be denoted by $R\langle T;\G,\D\rangle$, where $\G=\{\sigma_t\}_{t\in T}$ and $\D=\{\delta_t\}_{t\in T}$ are sequences
of automorphisms and skew derivations corresponding to generators $t\in T$. 
Since we do not assume that the mappings $t\mapsto \sigma_t$ and $t\mapsto\delta_t$ are injective, $\G$ and $\D$ are treated not as sets, but as sequences.
From the relations defining $R\langle T;\G,\D\rangle$ it follows immediately that every its element is a finite sum of monomials
of the form $r\omega$, where $r\in R^1$, $\omega$ is either a word in a finite set of elements from $T$ or $\omega=1$. 
Here $R^1$ denotes the usual extension  of $R$ to a ring with identity. If $R$ has no identity, then by $S^*$ we denote the set of finite sums of the form $r\omega$. Clearly, $S^*$ is an ideal of $S$.
The natural examples of skew extensions are:

\medskip

\noindent 1. The ring of skew polynomials $R[x;\sigma,\delta]$, and more generally the skew polynomial ring $R[X;\G,\D]$ over arbitrary set of non-commutative variables. More precisely, let $X$ be a nonempty set and suppose that to any $x\in X$ corresponds an automorphism $\si_x\in {\mathbf{Aut}}(R)$ and a $\si_x$-derivation 
$\de_x\colon R\to R$. Put $\G=\{\si_x\}_{ x\in X}$ and $\D=\{\de_x\}_{x\in X}$.  Then $R[X;\G,\D]$ is the ring
of all noncommutative polynomials 
$$
f(X)=\sum r_\Delta \Delta, 
$$
where $r_\Delta \in R$, and $\Delta\in \langle X\rangle$ - the free monoid generated by the set of free generators $X$.  The addition in $R[X; \G, \D]$ is defined as the addition of ordinary polynomials, and multiplication is given  subject to the rule:
$$
xr=\si_x(r)x+\de_x(r), \  \  \  {\rm where}\  \  \   x\in X,\  r\in R.
$$
We call the ring $R[X; \G, \D]$ the {\em free skew extension} of $R$. A special case when $\G=\{{\rm id}_R\}_{x\in X}$, that is when $\D$ is a sequence of ordinary derivations of $R$, was considered in \cite{TWC} under the name the Ore extension of derivation type. In \cite{BG3} the problem of semiprimitivity of free skew extensions is examined. In particular, it is proved there that if $R$ is a semiprime right Goldie ring, then the free skew extension $R[X; \G, \D]$ is always semiprimitive.
\medskip

\noindent 2. The crossed product $R\ast U(L)$, where  $U(L)$ is the universal enveloping algebra  of a Lie algebra 
$L$; see J.C. McConnel and J.C. Robson book \cite[section 1.7.12]{MR} for details.

\medskip

\noindent 3. The smash product $R\#H$, where $H$ is a Hopf algebra, which is  generated by grouplike and skew primitive elements. Note that the most typical nontrivial examples of
such Hopf algebras, which are neither group algebras nor universal enveloping
algebras, are given by quantized enveloping algebras $U_q(\mathfrak{g})$ for semisimple Lie algebras $\mathfrak{g}$ (see J. C. Brown and K. R. Goodearl book  \cite[Chapter I.6]{BrG}).

Recall that if $H$ is a Hopf algebra over a field $K$ with comultiplication $\Delta$, then $G(H)=\{\sigma\in H\mid \Delta(\sigma)=\sigma\otimes \sigma \}$ is the set of grouplike elements of $H$, and $P(H)=\bigcup\limits_{\sigma,\tau\in G(H)}\{x\in H\mid \Delta(x)=x\otimes \sigma+\tau\otimes x\}$ is the set of skew primitive elements of $H$. Observe that if $x$ is $(\sigma,\tau)$-primitive, then $\Delta(x\sigma^{-1})=x\sigma^{-1}\otimes 1+\tau\sigma^{-1}\otimes x\sigma^{-1}$. Thus if $R$ is an $H$-module algebra, then $x\sigma^{-1}$ acts on $R$ as a $\tau\sigma^{-1}$-derivation. Therefore, if $H$ is generated by $G(H)\cup P(H)$, then $H$ is also generated by the set $T$ of all grouplike and $(1,\sigma)$-primitive elements, where $\sigma\in G(H)$. In this case, every $h\in H$ is a linear combination
of products of elements of $T$. Thus
the smash product $R\#H$ is generated by $R\cup T$ and has a natural structure of a skew extension of $R$. 

For general facts about Hopf algebras and their actions, we refer the reader to S. Montgomery book \cite{M}.
\medskip  

Our main result is
\medskip

\noindent {\bf Theorem A.} {\it   Let $R\langle T,\G,\D\rangle$ be a skew extension  of an algebra $R$ over a commutative ring $K$ such that the family $\G$ is locally finite.
\begin{enumerate}
\item If $R^n\subseteq \mathcal{W}(R)$ for some $n>1$, where $\mathcal{W}(R)$ is the Wedderburn radical of $R$, then the $K$-algebra $R\langle T,\G,\D\rangle^*$ is locally nilpotent.

\item If the prime radical $P=P(R)$ of $R$ is strongly invariant under $\mathscr{D}$, then the ideal 
$P\langle T,\G,\D\rangle^*$ of $R\langle T,\G,\D\rangle$ is locally nilpotent.
\end{enumerate}}

The example presented in the introduction of \cite{BMS}, shows that a finiteness condition on the family $\G$ in the Theorem A is necessary. 

In Section 2 we apply Theorem A(2) to algebras over fields of characteristic zero  having locally nilpotent derivations. There is an abundance of literature related to kernels of locally nilpotent derivations. See for example \cite{F,N} and the references therein. The study of these kernels has turned out to be very fruitful in the solution of several problems in affine algebraic geometry. In Section 2 we consider non-commutative algebras
with locally nilpotent derivations having commutative kernels. We prove 
\medskip

\noindent {\bf Theorem B.} {\it Let $R$ be a $K$-algebra with a surjective locally nilpotent derivation $d$ such that $R$ is generated by $R_1=\ker d^2$. If $R^d$ is commutative, then the Jacobson radical $J(R)$ of $R$ is locally nilpotent. }

\section{Locally nilpotent skew extensions}
Throughout this section $K$ will be a commutative ring with $1$. For any $K$-algebra $R$, let $\mathbf{Aut}_K(R)$ be the group of all $K$-algebra automorphisms of $R$. Notice that any noncommutaive ring can be viewed as a $\mathbb{Z}$-algebra. An automorphism $\sigma\in \mathbf{Aut}_K(R)$ is said to be {\em locally finite}, if for any $v\in R$
the $K$-submodule generated by $\{\sigma^n(v)\mid n\geq 0\}$ is finitely generated.  For a subset $H$ of $\mathbf{Aut}_K(R)$ and a $K$-submodule $V$ of $R$, by $V(H)$ we denote the smallest $K$-submodule stable under any $\sigma\in H$. Observe that $V(H)=\{\sigma(v)\mid \sigma\in {\rm gr}(H),\   v\in V\}$,  where ${\rm gr}(H)$ is the subgroup of $\mathbf{Aut}_K(R)$ generated by $H$. A family $\G\subseteq \mathbf{Aut}_K(R)$ is said to be {\em locally finite},
if for any finite subset $G_0\subseteq \G$  and every finitely generated $K$-submodule $V$ of $R$, the $G_0$-stable 
$K$-submodule $V(G_0)=\{\sigma(v)\mid \sigma\in {\rm gr}(G_0), \  v\in V\}$ is finitely  generated. It is clear that if the group ${\rm gr}(\G)$ is locally finite, then the family $\G$ is locally finite. 

Let us fix a finite subset $T_0$ of $T$. Let $G_0$ and $D_0$ be the sequences
$$
G_0=\{\sigma_t\}_{t\in T_0} \  \  {\rm  and }\  \   D_0=\{\delta_t\}_{t\in T_0}.
$$

For $0\leq k\leq n$ let $\Delta_{n,k}$ be the set of all $K$-endomorphisms of $R$ of the form
$\eta_1\circ\eta_2\circ\dots\circ\eta_n$, where $\#\{i\mid \eta_i\in D_0\}=k$ and $\#\{i\mid \eta_i\in G_0\}=n-k$.
For a $K$-submodule $V$ of $R$ we put
$$
\Delta_{n,k}(V)=\{\varphi(v)\mid v\in V,\   \varphi\in \Delta_{n,k}\}.
$$

\begin{lem}
If $V$ is a finitely generated $K$-submodule of $R$, then $\Delta_{n,k}(V)$ is also a finitely generated $K$-module.
\end{lem}\label{fin_gen}
\begin{pf}
Notice that the set $\Delta_{n,k}$ is finite and
$$
\Delta_{n,k}(V)=\sum_{\varphi\in\Delta_{n,k}}\varphi(V).
$$
Hence $\Delta_{n,k}(V)$ is finitely generated as a finite sum of finitely generated $K$-modules.
\qed \end{pf}

\begin{lem}\label{V}
If $\G$ is a locally finite family of automorphisms of $R$, $T_0$ is a finite subset of $T$, and $V$ is a finitely generated $K$-submodule of $R$, then for any $k\geq 0$
$$
\sum_{n\geq k} \Delta_{n,k}(V)
$$
is contained in a finitely generated $G_0$-stable  $K$-submodule $\mathbf{V}_k$ of $R$.
\end{lem}
\begin{pf}
We will use the induction on $k$. For $k=0$, since the family $\G$ is locally finite, it is enough to put $\mathbf{V}_0=V(G_0)$. Suppose that $k>0$ and $\mathbf{V}_{k-1}$ is a finitely generated $G_0$-stable $K$-submodule of $R$ containing all modules of the form $\Delta_{n,k-1}(V)$, where $n\geq k-1$. Consider $W=\sum\limits_{\delta\in D_0}\delta(\mathbf{V}_{k-1})$. Then $W$ is a finitely generated $K$-module as a finite sum of homomorphic images of finitely generated modules. Put $\mathbf{V}_k= W(G_0)$. We will prove that for any $n\geq k$ $\Delta_{n,k}(V)\subseteq \mathbf{V}_k$. To this end, notice that if $\eta\in\Delta_{n,k}$, then there exist $\sigma_1,\dots\sigma_l\in G_0$, $\delta\in D_0$ and $\eta^\prime\in \Delta_{n^\prime,k-1}$ (where $n^\prime\geq k-1$) such that
$$
\eta=\sigma_1\circ\dots\circ\sigma_l\circ\delta\circ\eta^\prime.
$$
Now it is clear that 

$$
\begin{array}{rl}
\eta(V)& =(\sigma_1\circ\dots\circ\sigma_l\circ\delta)(\eta^\prime(V))\subseteq (\sigma_1\circ\dots\circ\sigma_l\circ\delta)(\mathbf{V}_{k-1})\\ \\
&= \sigma_1\circ\dots\circ\sigma_l(\delta(\mathbf{V}_{k-1}))\subseteq \sigma_1\circ\dots\circ\sigma_l(W)\subseteq W(G_0)=\mathbf{V}_k.
\end{array}
$$
This finishes the proof.
\qed \end{pf}

For a sequence $\ga=(a_1,a_2,\dots,a_m)$ of nonnegative integers let $s(\ga)=\sum\limits_{i=1}^{m} a_i$ and let $\mathscr{B}(\ga)$ be the set of all sequences of nonnegative integers $\gb=(b_1,b_2,\dots,b_m)$ such that

$$b_1+\dots+ b_k\leq a_1+\dots+a_k \  \  {\rm  for\  all\  }\  \ k=1,\dots, m.$$
 
For a sequence of integers $\x=(x_1,x_2,\dots,x_m)$  and $k\leq  m$, let $\widehat{x}_k=x_1+\dots+x_k$ .

\begin{lem}\label{product}
Let $V$ be  a $K$-submodule of $R$,  $T_0$ be a finite subset of $T$ with corresponding sets of automorphisms $G_0$ and skew derivations $D_0$. If   $\ga=(a_1,a_2,\dots,a_m)$ is a sequence of positive integers, then
 \bigskip
 \begin{equation}\label{*}
 T_0^{a_1}VT_0^{a_2}\dots VT_0^{a_m}V\subseteq \sum_{\gb\in \mathscr{B}(\ga)}\Delta_{{\widehat{a}_1,b_1}}(V)\Delta_{\widehat{a}_2-\widehat{b}_1,{b_2}}(V)\dots\Delta_{{\widehat{a}_m-\widehat{b}_{m-1},b_m}}(V)T_0^{s(\ga)-s(\gb)}
\end{equation}
\end{lem}
\begin{pf}
We will use the induction to prove this lemma. From the formula $tr=\sigma_t(r)t+\delta_t(r)$ it follows immediately that
\begin{equation}\label{m=1}
T_0^{a_1}V\subseteq\sum_{k=0}^{a_1}\Delta_{a_1,k}(V)T_0^{a_1-k},
\end{equation}
what proves (\ref{*}) for $m=1$. We assume that the inclusion (\ref{*}) holds for a given $m\geq 1$ and  arbitrary sequence $(a_1,\dots,a_m)$ of positive integers. Suppose that $\ga=(a_1,\dots,a_m,a_{m+1})$ and let $\ga^\prime=(a_1,\dots,a_m)$. Since $s(\ga^\prime)+a_{m+1}=s(\ga)$, using induction hypothesis and (\ref{m=1}) we obtain 
$$
\begin{array}{ll}
T_0^{a_1}VT_0^{a_2}  \dots VT_0^{a_{m+1}}V  \subseteq & \\ \\
\left(\sum\limits_{\gb^\prime\in \mathscr{B}(\ga^\prime)}\Delta_{{\widehat{a}_1,b_1}}(V)\Delta_{\widehat{a}_2-\widehat{b}_1,{b_2}}(V)\dots\Delta_{{\widehat{a}_m-\widehat{b}_{m-1},b_m}}(V)T_0^{s(\ga^\prime)-s(\gb^\prime)}\right) T_0^{a_{m+1}}V=&
\\ \\
\left(\sum\limits_{\gb^\prime\in \mathscr{B}(\ga^\prime)}\Delta_{{\widehat{a}_1,b_1}}(V)\Delta_{\widehat{a}_2-\widehat{b}_1,{b_2}}(V)\dots\Delta_{{\widehat{a}_m-\widehat{b}_{m-1},b_m}}(V)\right)T_0^{s(\ga)-s(\gb^\prime)}V\subseteq 
&\\ \\
\left(\sum\limits_{\gb^\prime \in \mathscr{B}(\ga^\prime)}\Delta_{{\widehat{a}_1,b_1}}(V)\Delta_{\widehat{a}_2-\widehat{b}_1,{b_2}}(V)\dots\Delta_{{\widehat{a}_m-\widehat{b}_{m-1},b_m}}(V)\right)\times &\\ \\
\left(\sum\limits_{k=0}^{s(\ga)-s(\gb^\prime)}\Delta_{s(\ga)-s(\gb^\prime),k}(V)T_0^{s(\ga)-s(\gb^\prime)-k} \right)=
&\\ \\

\sum\limits_{\gb \in \mathscr{B}(\ga)}\Delta_{{\widehat{a}_1,b_1}}(V)\Delta_{\widehat{a}_2-\widehat{b}_1,{b_2}}(V)\dots\Delta_{{\widehat{a}_m-\widehat{b}_{m-1},b_m}}(V)\Delta_{{\widehat{a}_{m+1}-\widehat{b}_{m},b_{m+1}}}(V)T_0^{s(\ga)-s(\gb)}.
\end{array}
$$
This finishes the proof.
\qed \end{pf}

For any ring $R$, let $\mathcal{W}(R)$ be the sum of all nilpotent ideals of $R$. $\mathcal{W}(R)$ is called the Wedderburn radical of $R$. Recall that S.A. Amitsur proved in \cite{A} that if $R$ is a PI-ring of degree $d$, and $S$ is a nil subring of $R$, then $S^{m}\subseteq \mathcal{W}(R)$, where $m\leq \lfloor d/2\rfloor$.Therefore, the assumptions of our first main result seem to be natural.

\begin{thm}\label{nilp}
Let $R\langle T;\G,\D\rangle$ be a skew extension  of a $K$-algebra $R$ such that the family $\G$ is locally finite. Suppose that $R^n\subseteq \mathcal{W}(R)$ for some $n>1$. Then the $K$-algebra $R\langle T;\G,\D\rangle^*$ is locally nilpotent.
\end{thm}
\begin{pf}
Notice that any finite subset of $R\langle T,\G,\D\rangle^*$ is contained in 
$$
\sum_{i=0}^N VT_0^i,
$$
where $N\geq 0$,  $V$ is a finitely generated $K$-submodule of $R$, $T_0$ is a finite subset of $T$ with corresponding finite sequences of automorphisms $G_0=\{\sigma_t\}_{t\in T_0}$ and skew derivations $D_0=\{\delta_t\}_{t\in T_0}$.
We need  to prove that 
$$
\left(\sum\limits_{i=0}^N VT_0^i\right)^l=0,
$$ 
for some $l>1$.
Consider the family $$
\mathscr{F}=\{\mathbf{V}_{k_1}\cdot\mathbf{V}_{k_2}\cdot \ldots\cdot\mathbf{V}_{k_n}\mid k_1+k_2+\ldots+k_n\leq 2nN \},
$$
where $\mathbf{V}_{k_i}$ is a finitely generated $K$-submodule of $R$ containing all $\Delta_{j,k_i}$ for $j\geq k_i$ (defined in Lemma \ref{V}.) It is clear that the family $\mathscr{F}$ is finite and any member of $\mathscr{F}$ is finitely generated as a $K$-module. Since $R^n\subseteq \mathcal{W}(R)$, we obtain that
$
\bigcup \mathscr{F}\subseteq \mathcal{W}(R)
$.
This forces that there exists a nilpotent ideal $I$ of $R$ containing any $K$-module from $\mathscr{F}$. Let $s$ be such that $I^s=0$. We will prove that
$$
\left(\sum\limits_{i=0}^N VT_0^i\right)^{2ns}=0.
$$ 
To this end it suffices to show that
$$
T_0^{a_1}VT_0^{a_2}\ldots VT_0^{a_m}V=0,
$$
where $m=2ns$ and  $\ga=(a_1,a_2,\dots,a_m)$ is a sequence  of integers such that $0\leq a_k\leq N$.
Notice that for any sequence $\gb=(b_1,b_2,\dots,b_m)\in \mathscr{B}(\ga)$
$$
\Delta_{{\widehat{a}_1,b_1}}(V)\Delta_{\widehat{a}_2-\widehat{b}_1,{b_2}}(V)\dots\Delta_{{\widehat{a}_m-\widehat{b}_{m-1},b_m}}(V)\subseteq \mathbf{V}_{b_1}\cdot\mathbf{V}_{b_2}\cdot \ldots\cdot\mathbf{V}_{b_m}
= 
$$
$$
(\mathbf{V}_{b_1}\cdot\mathbf{V}_{b_2}\cdot \ldots\cdot\mathbf{V}_{b_n})(\mathbf{V}_{b_{n+1}}\cdot\mathbf{V}_{b_{n+2}}\cdot \ldots\cdot\mathbf{V}_{b_{2n}})\ldots (\mathbf{V}_{b_{(2s-1)n+1}}\cdot\mathbf{V}_{b_{(2s-1)n+2}}\cdot \ldots\cdot\mathbf{V}_{b_{2sn}}).
$$
Notice that among $2s$ disjoint consecutive products  of the form $\mathbf{V}_{b_{kn+1}}\cdot\mathbf{V}_{b_{kn+2}}\cdot \ldots\cdot\mathbf{V}_{b_{(k+1)n}}$ (where $k=0,1,\dots 2s-1$) at least $s$ of those satisfy the inequality 
$$
b_{kn+1}+b_{kn+2}+\ldots+b_{(k+1)n}\leq 2nN
$$
Otherwise, there are $s+1$ disjoint products satisfying $b_{kn+1}+b_{kn+2}+\ldots+b_{(k+1)n}> 2nN$, what implies that
$$
mN\geq a_1+a_2+\ldots+a_m\geq b_1+b_2+\ldots+b_m >(s+1)2nN,
$$
a contradiction with $m=2ns$. Thus at least $s$ products $\mathbf{V}_{b_{kn+1}}\cdot\mathbf{V}_{b_{kn+2}}\cdot\ldots\cdot\mathbf{V}_{b_{(k+1)n}}$ are contained in $I$, so
$$
\Delta_{{\widehat{a}_1,b_1}}(V)\Delta_{\widehat{a}_2-\widehat{b}_1,{b_2}}(V)\dots\Delta_{{\widehat{a}_m-\widehat{b}_{m-1},b_m}}(V)\subseteq I^s=0.
$$
Now Lemma \ref{product} finishes the proof.
\qed \end{pf}

Since any nil PI-algebra $R$ satisfies assumptions of Theorem \ref{nilp}, we obtain

\begin{cor}
Let $R$ be a nil $K$-algebra satisfying a polynomial identity. Then any free skew extension $R[X;\G,\D]$ with
a locally finite  family $\G$ of automorphisms is locally nilpotent.
\end{cor}
\medskip

A useful property of the prime radical of $R$ is that it can also be defined by transfinite induction as the union of an ascending chain of ideals $P_\alpha \subseteq R$
as follows:
\begin{enumerate}
	\item  $P_0=0 $, $P_1=\mathcal{W}(R)$;
	\item  $P_{\alpha+1}$ is the ideal of $R$ such that $\mathcal{W}(R/P_{\alpha})= P_{\alpha+1}/P_{\alpha}$;
	\item  if $\alpha$ is a limit ordinal, then $P_{\alpha}=\bigcup\limits_{\beta<\alpha}P_\beta$.
\end{enumerate}
\vskip.1in

Observe that each $P_\alpha$ is stable under any automorphism $\sigma\in \mathbf{Aut}_K(R)$ and  there exists an ordinal $\gamma$ such that $P_\gamma=P_{\gamma+1} = P(R)$. We will say that the prime radical $P(R)$ is {\em strongly invariant} under  a family  of skew derivations $\mathscr{D}$ if 
$$
\delta(P_\alpha)\subseteq P_\alpha \  \  {\rm  for\  any \ ordinal }\  \   \alpha \  \ {\rm and }\  \  \delta\in\mathscr{D}.
$$

\noindent{\bf Remark.} By  \cite[Theorem 6]{BG2} the prime radical of an algebra $R$ over a field $K$ is strongly invariant under
\begin{enumerate}
\item all derivations, provided char\,$K=0$;
\item $q$-skew $\sigma$-derivations $\delta$; that is when $\delta\sigma=q\sigma\delta$, $q\in K^\times$, $\sigma$ is locally finite and $1+q+q^2+\dots+q^n\neq 0$ for all $n>0$.
\end{enumerate}

\medskip

\begin{thm}\label{prime_rad}
Let $R\langle T,\G,\D\rangle$ be a skew extension  of a $K$-algebra $R$ such that the family $\G$ is locally finite, and the prime radical $P=P(R)$ is strongly invariant under $\mathscr{D}$. Then the ideal $P\langle T,\G,\D\rangle^*$ of $R\langle T,\G,\D\rangle$ is locally nilpotent.
\end{thm}
\begin{pf}
We will show using transfinite induction that for any  ordinal $\alpha$, the ideal $P_\alpha\langle T,\G,\D\rangle^*$ is locally nilpotent. The case $\alpha=0$ is clear. Let $\alpha>0$ be an ordinal and suppose that the result is true for all $\beta<\alpha$. First consider the case when $\alpha$ is a non limit ordinal, i.e. $\alpha=\beta+1$ for some ordinal $\beta$. Let $A$ be a finite subset of $P_{\beta+1}\langle T;\G,\D\rangle^*$. Take a finitely generated  $K$-module $V\subseteq P_{\beta+1}$  and a finite subset $T_0\subseteq T$ such that $A\subseteq \sum\limits_{i=0}^NVT_0^i$. Our aim is to prove that there exists $j>1$ such that
$$
\left(\sum\limits_{i=0}^NVT_0^i\right)^j=0
$$
Let $G_0=\{\sigma_t\}_{t\in T_0}$ and $D_0=\{\delta_t\}_{t\in T_0}$ be  sequences of automorphisms and skew derivations corresponding to $T_0$. 
Similarly as in the proof of Theorem \ref{nilp} consider the family
$$
\mathscr{F}=\{\mathbf{V}_k\mid k\leq 2N\},
$$
where $\mathbf{V}_k$ are finitely generated $K$-modules defined in Lemma \ref{V}.
Since $P_{\beta+1}$ is stable under automorphisms and skew derivations, we conclude that for any $k\geq 1$, $\mathbf{V}_k\subseteq P_{\beta+1}$ and $\mathbf{V}_k$ is a $G_0$-stable finitely generated $K$-module containing all $\Delta_{n,k}(V)$, where $n\geq k$.
Thus we can find an ideal $I$ and integer $s>1$ such that $\bigcup\mathscr{F}\subseteq I$ and $I^s\subseteq P_\beta$.
We will show that 
$$
\left(\sum\limits_{i=0}^NVT_0^i\right)^{2s}\subseteq P_{\beta}\langle T;\G,\D\rangle^*.
$$

Let $\mathscr{A}$ be the set of all sequences $\ga=(a_1,a_2,\dots,a_{2s})$  of integers such that $0\leq a_k\leq N$. Take $\ga\in \mathscr{A}$.
Notice that for any sequence $\gb=(b_1,b_2,\dots,b_{2s})\in \mathscr{B}(\ga)$ the product $\mathbf{V}_{b_1}\cdot\mathbf{V}_{b_2}\cdot \ldots\cdot\mathbf{V}_{b_{2s}}$ contains at least $s$ terms from the family $\mathscr{F}$. Otherwise, the sequence
$\gb$ contains at least $s+1$ numbers satisfying inequality $b_i>2N$. Thus
$$
2sN\geq a_1+a_2+\ldots+a_{2s}\geq b_1+b_2+\ldots+b_{2s}>(s+1)2N,
$$
a contradiction. Consequently,
$$
\Delta_{{\widehat{a}_1,b_1}}(V)\Delta_{\widehat{a}_2-\widehat{b}_1,{b_2}}(V)\ldots\Delta_{{\widehat{a}_{2s}-\widehat{b}_{2s-1},b_{2s}}}(V)\subseteq \mathbf{V}_{b_1}\cdot\mathbf{V}_{b_2}\cdot \ldots\cdot\mathbf{V}_{b_{2s}}\subseteq I^s
\subseteq P_{\beta}.
$$
Putting $W=\sum\limits_{\ga\in \mathscr{A}}\sum\limits_{\gb\in \mathscr{B}(\ga)}\mathbf{V}_{b_1}\cdot\mathbf{V}_{b_2}\cdot \ldots\cdot\mathbf{V}_{b_{2s}}$
we see that $W\subseteq P_\beta$ is finitely generated as a $K$-module and there exists an integer $m>1$ such  that
$$
\left(\sum\limits_{i=0}^NVT_0^i\right)^{2s}\subseteq \sum\limits_{j=0}^mWT_0^j\subseteq P_{\beta}\langle T;\G,\D\rangle^*.
$$
Therefore by the induction hypothesis there exists $l>1$ such that $(\sum\limits_{j=0}^mWT_0^j)^l=0$, so
$$
\left(\sum\limits_{i=0}^NVT_0^i\right)^{2sl}=0.
$$

If $\alpha$ is a limit ordinal, then any finite subset $A\subseteq P_\alpha\langle T;\G,\D\rangle^*$ is contained in 
$P_\beta\langle T,\G,\D\rangle^*$ for some $\beta<\alpha$. Thus by the induction hypothesis, $A$ generates a nilpotent subalgebra. This finishes the proof.
\qed \end{pf}
\begin{cor}
Let $R$ be an algebra over a field $K$ and let $R[X;\G,\D]$ be a free skew extension such that
\begin{enumerate}
\item the family $\G=\{\sigma_x\}_{x\in X}$ is locally finite,
\item for any $x\in$ there exists $q_x\in K^\times$ such that $\delta_x\sigma_x=q_x\sigma_x\delta_x$, 
\item  for any $x\in X$ and $n>0$,  $1+q_x+\dots+q_x^n\neq 0$.
\end{enumerate}
Then $P[X;\G,\D]$ is a locally nilpotent ideal of $R[X;\G,\D]$.
\end{cor}

In \cite[Theorem 2.4]{BG3} it is proved that if $R$ is a semiprime Goldie ring, then for any  set $X$ of noncommuting variables and sequences $\G=\{\sigma_x\}_{x\in X}$, $\D=\{\delta_x\}_{x\in X}$ of automorphisms and skew derivations of $R$, the free skew extension $R[X; \G,\D]$ is semiprimitive. As a consequence we obtain

\begin{cor}
Let $R$ be an algebra over a commutative ring $K$ such that the factor algebra $R/P(R)$ is semiprime Goldie. Let  $X$ be a set of noncommuting variables and  $\G=\{\sigma_x\}_{x\in X}$, $\D=\{\delta_x\}_{x\in X}$ be sequences of automorphisms and skew derivations of $R$ such that $\G$ is locally finite and $P=P(R)$ is strongly invariant under $\D$. Then the Jacobson radical
$J(R[X; \G,\D])$ is equal to $P[X; \G,\D]$. In particular,  $J(R[X; \G,\D])$ is locally nilpotent.
\end{cor}
\begin{pf}
We conclude from Theorem \ref{prime_rad} that $P[X;\G,\D]$ is a locally nilpotent ideal of 
the free skew extension $R[X;\G,\D]$. In particular,  $P[X;\G,\D]\subseteq J(R[X;\G,\D])$. Observe that 
$$
R[X;\G,\D]/P[X;\G,\D] \simeq
(R/P)[X;\overline{\G},\overline{\D}],
$$ 
where $\bar{\sigma}_x$, $\bar{\delta}_x$ are induced automorphisms and skew derivations of $R/P$. The algebra $R/P$ is semiprime Goldie, so by Theorem 2.4 of \cite{BG3} the free skew extension $(R/P)[X;\overline{\G},\overline{\D}]$ is semiprimitive. It means that
$J(R[X;\G,\D])=P[X;\G,\D]$. 
\qed \end{pf}

\section{Locally nilpotent derivations}

We say that a derivation $d$ of a ring $R$ is locally nilpotent, if for each $r \in R$, there exists $n = n(r) \geq 1$ such that $d^n(r) = 0$.
It then follows that if we let $R_n = \ker d^{n+1}= \{r \in R \mid d^{n+1}(r) = 0\}$, then
$$
R_0 \subseteq R_1 \subseteq R_2 \subseteq \cdots ,
$$
$R = \bigcup\limits_{n \geq 0} R_n$, $R_0=R^d$ is the invariants of $d$, and $R_1$ is the kernel of $d^2$.
By the degree of a non-zero element $a\in R$, which we denote as $\deg (a)$, we mean the integer $n$ such that $a \in R_n\setminus R_{n-1}$.
\medskip

In this section we  apply Theorem \ref{prime_rad} to algebras over fields of characteristic zero  having locally nilpotent derivations. In  \cite{BG1} we examined the Jacobson radical of algebras with locally nilpotent skew derivations in the case when the subalgebra of invariants has no non-zero nil ideals.
\medskip

Observe that if $R$ is a $K$-algebra with $1$ (char\,$K=0$), then any differential polynomial algebra $R[x;\delta]$
has a locally nilpotent derivation $d=\frac{d}{dx}$; that is $d(x)=1$ and $d(r)=0$ for all $r\in R$. In this example,
$R[x;\delta]^d=R$, $d$ is surjective and $R[x;\delta]$ is generated by $\ker d^2=R+Rx$. Another interesting example
of a surjective locally nilpotent derivation of a non-unital algebra we obtain by using the Grassmann algebra.
\begin{example}{\rm
Let $E$ be the Grassmann algebra over $\mathbb{Q}$ generated
by an infinite set $\{e_i\mid i>0\}$ with relations $e_ie_j+e_je_i=0$, for all $i,j>0$. It is well known that $E=\mathcal{W}(E)$ and $E$
does not satisfy a polynomial identity. The mapping  given by $d(e_{i+1})=e_i$ and $d(e_1)=0$ extends uniquely to a locally nilpotent derivation of $E$. We will show that $d$ is a surjection. For $n\geq 1$ let $B_n=\{e_{i_1}e_{i_2}\dots e_{i_n}\mid i_1<i_2<\dots< i_n\}$. Then $B=\bigcup\limits_{n\geq 1} B_n$ is a linear basis of $E$.
Hence it is enough to show that $B_n\subseteq \im d$. By the definition it follows that $B_1\subseteq \im d$.  Suppose that there exists
$n>1$ such that $B_n\not\subset \im d$. Observe that $B_n$ is well ordered with respect to lexicographical order $\prec\colon$
$$
e_{i_1}e_{i_2}\dots e_{i_n}\prec e_{j_1}e_{j_2}\dots e_{j_n} \  \  {\rm  if\  and \  only\   if\  for \  some  } \  \ k\leq n, \  \ i_k<j_k 
\  {\rm  and }\  \   i_s=j_s \  \  {\rm  for \ all } \  \  s<k.
$$
Since $d(e_1e_2\dots e_{n-1}e_{n+1})= e_1e_2\dots e_{n-1}e_n$, we have that $e_1e_2\dots e_{n-1}e_n\in \im d$.  Take the smallest element $e_{i_1}e_{i_2}\dots e_{i_n}$ in $B_n$ (with respect to $\prec$) such that $e_{i_1}e_{i_2}\dots e_{i_n}\not\in \im d$. Notice that
\begin{equation}\label{gr}
d(e_{i_1}e_{i_2}\dots e_{i_{n-1}}e_{i_n+1})=\sum_{k=1}^{n-1}e_{i_1}\dots d(e_{i_k})\dots e_{i_n+1}+e_{i_1}e_{i_2}\dots e_{i_n}.
\end{equation}
Since $e_{i_1}\dots d(e_{i_k})\dots e_{i_n+1}\prec e_{i_1}e_{i_2}\dots e_{i_n}$, by our hypothesis we obtain $e_{i_1}\dots d(e_{i_k})\dots e_{i_n+1}\in \im d$. Therefore, by equation (\ref{gr}) $e_{i_1}e_{i_2}\dots e_{i_n}\in \im d$, a contradiction. Thus $d$ is a surjective locally nilpotent derivation of the Grassmann algebra $E$.
}\end{example}

\begin{prop}\label{loc_der}
Let $d$ be a locally nilpotent derivation of an algebra $R$ over a field $K$ of characteristic zero. If $R^d$ is commutative, then $J(R)\cap R^d\cap d(R)$ is a nil ideal of $R^d$. In particular, if $d$ is surjective, then $J(R)\cap R^d\subseteq P(R^d)$.
\end{prop}
\begin{pf}
Take $a\in J(R))\cap R^d \cap d(R)$ and choose $x\in R$ such that $d(x)=a$. Since $xa\in J(R)$, we can find $b\in R$
such that
\begin{equation}\label{quasi_inv}
xa+b=(xa)b =b(xa).
\end{equation}

First suppose that there exists $n>0$ such that $a^nb=0$. Then by (\ref{quasi_inv}) it follows that $a^nxa=0$. Aplying $d$ we obtain $0=d(a^nxa)=a^nd(x)a= a^{n+2}$. Hence $a$ is nilpotent. Since $R^d$ is commutative, we obtain that $a\in P(R^d)$.

Now suppose that $a^nb\neq 0$ for any $n>0$. Take
 $k,n \geq 1$ be such that
$$
k-1=\min\{\deg(a^mb)\mid m\in \mathbb{N}\}=\deg(a^nb). 
$$
Then $d^{k-1}(a^nb)\neq0$ and $d^k(a^nb)=0$. By (\ref{quasi_inv}) it follows that
\begin{equation}\label{mult}
a^nxa+a^nb=a^nbxa
\end{equation}
Hence if $k\geq 2$, then $d^k(x)=0$ and by (\ref{mult})
$$
0=d^k(a^nxa)+d^k(a^nb)=d^k(a^nbxa)=kd^{k-1}(a^nb)d(x)a=kd^{k-1}(a^nb)a^2.
$$
Since $d^{k-1}(a^nb)\in R^d$ and $R^d$ is commutative, we obtain 
$$
0=d^{k-1}(a^nb)a^2=a^2d^{k-1}(a^nb)=d^{k-1}(a^{n+2}b),
$$
a contradiction with the definition of $k$. Consequently $k=1$, $a^nb\in R^d$ and by (\ref{mult})
$$
a^{n+2}= d(a^nxa)+d(a^nb)=d(a^nbxa)=a^nbd(x)a=(a^nb)a^2=a^{n+2}b.
$$
Thus $a^{n+2}=a^{n+2}b$ and  by (\ref{mult})
$$
a^{n+2}xa+a^{n+2}b=a^{n+2}bxa,
$$
so $a^{n+2}=0$, a contradiction with our assumption that $a^nb\neq 0$ for all $n$. This finishes the proof.
\qed \end{pf}

\medskip

A subspace $V$ of an algebra $R$ is said to be an {\it invariant subspace}
if $\sigma(V)\subseteq V$ for every $\sigma\in \mathbf{Aut}_K(R)$.

\begin{lem}\label{key} Every invariant subspace of an algebra $R$ over a field of characteristic zero is stable under locally nilpotent derivations of $R$. In particular, the Jacobson radical $J(R)$ is stable under locally nilpotent derivations.
\end{lem}

\begin{pf} Let $d$ be a locally nilpotent derivation of $R$.
It is well known that the exponential map
$\exp(d)=\sum\limits_{j=0}^{\infty}\frac{d^j}{j!}$
is an algebra automorphism of $R$. Thus if $V$ is an invariant subspace of $R$, then $\exp(d)(V)\subseteq V$. Let $x\in V$ and take an integer $m$  such that $d^m(x)\neq 0$ and $d^{m+1}(x)=0$. Then
$
\exp(d)(x)=\sum\limits_{j=0}^{m}\frac{d^j(x)}{j!}.
$
We claim that for $k=0,1,\dots,m$ there exist positive rational numbers $c^{(k)}_k, c^{(k)}_{k+1},\dots,c^{(k)}_m$ such that
\begin{equation}\label{1}
 x_k=\sum\limits_{j=k}^mc^{(k)}_jd^j(x)\in V.
\end{equation}
For $k=0$, since $\exp(d)(x)\in V$, it suffices to put
 $c^{(0)}_j=1/j!$ (for $j=0,1,\dots,m$). Let $0\leq k< m$ and $x_k=\sum\limits_{j=k}^mc^{(k)}_jd^j(x)\in V$ for some positive rational numbers
 $c^{(k)}_j.$ Then the element $\exp(d)(x_k)-x_k$ belongs to $V.$ Notice that

$$
\exp(d)(x_k)-x_k =\sum\limits_{i=1}^m\frac{d^i(x_k)}{i!}=
\sum\limits_{i=1}^m\sum\limits_{j=k}^m\frac{c^{(k)}_j}{i!}d^{i+j}(x)=
\sum\limits_{s=k+1}^m\left(\sum\limits_{i=1}^{s-k} \frac{c^{(k)}_{s-i}}{i!}\right)d^s(x).
$$
The elements $c^{(k+1)}_s=\sum\limits_{i=1}^{s-k}\frac{c^{(k)}_{s-i}}{i!}\in \mathbb{Q}$ (for $s=k+1,\dots,m$)  and $x_{k+1}=\sum\limits_{s=k+1}^mc^{(k+1)}_sd^s(x)$ satisfy desired property, which proves the claim. Since $x_m=c^{(m)}_md^m(x)\in V$ and $c^{(m)}_m>0$, one obtains $d^m(x)\in V$.
By induction it follows immediately from (\ref{1}) that $d^s(x)\in V$ for all  $s\geq 0$. Thus the subspace $V$ is $d$-stable.
\qed \end{pf}
\begin{thm}
Let $R$ be a $K$-algebra with a surjective locally nilpotent derivation $d$ such that $R$ is generated by $R_1=\ker d^2$. If $R^d$ is commutative, then the Jacobson radical $J(R)$ of $R$ is locally nilpotent.
\end{thm}
\begin{pf}
Observe that for any $t\in R_1$ and $r\in R^d$ the element $\delta_t(r)=tr-rt$ is a constant of $d$. Thus 
$\D=\{\delta_t\mid t\in R_1\}$ is a family of derivations of $R^d$. By assumption $R$ is generated by $T=R_1$, so
$R=R^d\langle T;\G,\D\rangle$ is a skew extension of $R^d$, with trivial family  of automorphisms $\G=\{{\rm id}_R\}$. Let $P$ be the prime radical of $R^d$. By Theorem \ref{prime_rad} the ideal $P\langle T;\G,\D\rangle^*$ is locally nilpotent, so
$
P\langle T;\G,\D\rangle^*\subseteq J(R).
$
We will prove that
\begin{equation}\label{inclusion}
J(R)^2\subseteq P\langle T;\G,\D\rangle^*.
\end{equation}
To this end suppose that $J(R)^2\not\subseteq P\langle T;\G,\D\rangle^*$ and take $a,b\in J(R)$ such that $ab\not\in P\langle T;\G,\D\rangle^*$ and $\deg(a)+\deg(b)$ is the smallest possible. By Lemma \ref{key} $d(a)$ and $d(b)$ belong to $J(R)$ and clearly $\deg(d(a))<\deg(a)$, $\deg(d(b))<\deg(b)$. Therefore $d(a)b\in P\langle T;\G,\D\rangle^*$ and
$ad(b)\in P\langle T;\G,\D\rangle^*$. Since $d$ is surjective there exist $a_1,\dots, a_n, b_1,\dots,b_m\in P$
and $x_1,\dots,x_n,y_1,\dots,y_m\in R$ such that
$$
d(a)b=\sum_{i=1}^n a_id(x_i)  \  \  {\rm and }\  \   ad(b)=\sum_{j=1}^m b_jd(y_j).
$$
Thus
$$
d(ab)=d(a)b+ad(b)=\sum_{i=1}^n a_id(x_i)+\sum_{j=1}^m b_jd(y_j) =d(\sum_{i=1}^n a_ix_i+\sum_{j=1}^m b_jy_j)
$$
and hence $ab-\sum\limits_{i=1}^n a_ix_i-\sum\limits_{j=1}^m b_jy_j\in R^d\cap J(R)$. But Proposition \ref{loc_der}
implies that $J(R)\cap R^d\subseteq P$, so $ab\in P\langle T;\G,\D\rangle^*$. The obtained contradiction finishes the proof of (\ref{inclusion}). It also  implies immediately that the Jacobson radical $J(R)$ is locally nilpotent.
\qed \end{pf}

\section*{Acknowledgements} The author would like to thank the anonymous referee for his helpful comments. 


\end{document}